\documentclass[oneside,english]{amsart}
\usepackage[T1]{fontenc}
\usepackage[latin9]{inputenc}
\usepackage{babel}
\usepackage{amstext}
\usepackage{amsthm}
\usepackage{amssymb}
\usepackage[unicode=true,pdfusetitle,
 bookmarks=true,bookmarksnumbered=false,bookmarksopen=false,
 breaklinks=false,pdfborder={0 0 1},backref=false,colorlinks=false]
 {hyperref}

\makeatletter
\numberwithin{equation}{section}
\theoremstyle{plain}
\newtheorem{thm}{\protect\theoremname}[section]
\theoremstyle{plain}
\newtheorem{cor}{\protect\corollaryname}[section]
\theoremstyle{plain}
\newtheorem{lem}{\protect\lemmaname}[section]

\makeatother

\providecommand{\corollaryname}{Corollary}
\providecommand{\lemmaname}{Lemma}
\providecommand{\theoremname}{Theorem}

\begin{document}
\title[Hecke-type identities]{Hecke-type identities associated with definite quadratic forms}
\author{Bing He}
\address{School of Mathematics and Statistics, Central South University \\
Changsha 410083, Hunan, People's Republic of China}
\email{yuhe001@foxmail.com; yuhelingyun@foxmail.com}
\begin{abstract}
Since the study by Jacobi and Hecke, Hecke-type series have received
a lot of attention. Unlike such series associated with indefinite
quadratic forms, identities on Hecke-type series associated with definite
quadratic forms are quite rare in the literature. Motivated by the
works of Liu, we first establish many parameterized identities with
two parameters by employing different $q$-transformation formulas
and then deduce various Hecke-type identities associated with definite
quadratic forms by specializing the choice of these two parameters.
As applications, we utilize some of these Hecke-type identities to
establish families of inequalities for several partition functions.
Our proofs heavily rely on some formulas from the work of Zhi-Guo
Liu \cite{L6}.
\end{abstract}

\maketitle

\section{Introduction}

Hecke-type series are of the type 
\[
\sum_{(m,n)\in D}(-1)^{H(m,n)}q^{Q(m,n)+L(m,n)},
\]
where $H$ and $L$ are linear forms, $Q$ is a quadratic form and
$D$ is some subset of $\mathbb{Z}\times\mathbb{Z}$ such that $Q(m,n)\geq0$
for each $(m,n)\in D.$ We allow $Q(m,n)$ to be definite here although
$Q(m,n)$ was assumed to be indefinite historically. The following
classical identity, which expresses an infinite product as a Hecke-type
series, is due to Jacobi \cite[(3.15)]{A84}: 
\[
(q;q)_{\infty}^{3}=\sum_{n=-\infty}^{\infty}\sum_{m\geq|n|}(-1)^{m}q^{m(m+1)/2},
\]
where 
\[
(q;q)_{\infty}:=\prod_{k=1}^{\infty}(1-q^{k}).
\]
Here and in the sequel, $|q|<1$ is assumed. Motivated by this identity,
Hecke \cite{H} investigated theta series relating indefinite quadratic
forms. In particular, Hecke \cite[p. 425]{H} presented the following
identity, which is due to Rogers \cite[p. 323]{R}:
\[
(q;q)_{\infty}^{2}=\sum_{n=-\infty}^{\infty}\sum_{|m|\leq n/2}(-1)^{m+n}q^{(n^{2}-3m^{2})/2+(m+n)/2}.
\]

Since the study by Hecke, this type of identities have attracted broad
interest among many mathematicians. For instance, Kac and Peterson
\cite{KP} showed ways to prove Hecke-type identities by using affine
Lie algebra. Andrews \cite{A86,A88} linked them to the fifth and
seventh order mock theta functions. In \cite{Hi} Hickerson applied
Hecke-type identities in his beautiful proof of the mock theta conjecture.
In addition, Zwegers \cite{Z} developed a theory of transformation
formulas for Hecke-type identities arising from the mock theta functions.
However, unlike such series associated with indefinite quadratic forms,
results on Hecke-type identities associated with definite quadratic
forms are quite few in the literature. Using a $q$-transformation
formula, Liu \cite[(4.12)--(4.14)]{L6}\footnote{The factor $(-1)^{n}$ should be deleted on the right-hand side of
\cite[(4.12) and (4.14)]{L6} while the factor $(-1)^{n}$ is missing
on the right-hand side of \cite[(4.13)]{L6}.} established three interesting Hecke-type identities associated with
definite quadratic forms:
\begin{align*}
 & \text{\ensuremath{\sum_{n=0}^{\infty}\frac{q^{n^{2}}}{(q^{2};q^{2})_{n}}}}=\frac{1}{(q;q)_{\infty}}\sum_{n=0}^{\infty}\sum_{j=-n}^{n}(-1)^{j}(1-q^{2n+1})q^{n^{2}+j^{2}},\\
 & \sum_{n=0}^{\infty}\frac{(-1)^{n}q^{n(n-1)/2}}{(-q;q)_{n}}=\sum_{n=0}^{\infty}\sum_{j=-n}^{n}(-1)^{n+j}(1-q^{2n+1})q^{n(n-1)/2+j^{2}},\\
 & \sum_{n=0}^{\infty}\frac{(-1)^{n}q^{n(n-1)/2}}{(q;q)_{n}}=\frac{(-q;q)_{\infty}}{(q;q)_{\infty}}\sum_{n=0}^{\infty}\sum_{j=-n}^{n}(-1)^{j}(1-q^{2n+1})q^{n(n-1)/2+j^{2}},
\end{align*}
where 
\[
(a;q)_{n}:=\prod_{k=0}^{n-1}(1-aq^{k}).
\]
In \cite{WY} Wang and Yee found one identity of this type:
\begin{equation}
\sum_{n=1}^{\infty}\frac{q^{n}(q;q^{2})_{n}}{(1+q^{2n})(-q;q^{2})_{n}}=\sum_{n=0}^{\infty}\sum_{j=-n+1}^{n}(-1)^{j}q^{n^{2}+j^{2}}.\label{eq:6-1}
\end{equation}
Two similar identities 
\begin{align*}
1+2\sum_{n=1}^{\infty}\frac{q^{n^{2}+n}}{(1+q^{n})(q;q)_{n}} & =\frac{1}{(q;q)_{\infty}}\sum_{n=0}^{\infty}\sum_{j=-n}^{n}(-1)^{j}q^{(3n^{2}+n)/2+j^{2}},\\
1+2\sum_{n=1}^{\infty}\frac{(-1)^{n}q^{n(n+1)/2}}{1+q^{n}} & =\sum_{n=0}^{\infty}\sum_{j=-n}^{n}(-1)^{n+j}(1-q^{2n+1})q^{n^{2}+j^{2}}
\end{align*}
were recently discovered by Chan and Liu \cite[(4.8) and (4.9)]{CL}.
To the best of our knowledge, these identities may be the only six
Hecke-type identities associated with definite quadratic forms in
the literature. Finding such identities becomes an interesting and
important topic.

In \cite{L5,L1,L6}, using two general $q$-transformation formulas,
Liu discovered many amusing Hecke-type identities although most of
his identities are only associated with indefinite quadratic forms.
Motivated by the works of Liu, we shall present a lot of Hecke-type
identities associated with definite quadratic forms. Namely, we first
set up various parameterized identities with two parameters by using
different $q$-transformation formulas and then derive many Hecke-type
identities associated with definite quadratic forms by specializing
the choice of these two parameters.

The rest of this paper is organized as follows. In Section \ref{sec:Pre}
we recall some formulas from the theory of basic hypergeometric series.
In Section \ref{sec:3} we establish many $(a,b)$-parameterized identities
and then employ these parameterized identities to deduce various Hecke-type
identities associated with definite quadratic forms. For example,
\begin{align*}
\sum_{n=1}^{\infty}\frac{q^{n}}{1+q^{2n}} & =\sum_{n=1}^{\infty}\sum_{j=-n+1}^{n}q^{n^{2}+j^{2}},\\
\sum_{n=1}^{\infty}\frac{(-1)^{n}q^{n^{2}+n}}{(1+q^{2n})(q;q^{2})_{n}} & =\sum_{n=1}^{\infty}\sum_{j=-n+1}^{n}(-1)^{n}q^{2n^{2}+j^{2}}.
\end{align*}
Both of these two identities are associated with definite quadratic
forms $n^{2}+j^{2}$ and $2n^{2}+j^{2}$ respectively.

Finally, as applications, we use several Hecke-type identities in
Section \ref{sec:3} to establish inequalities for certain partition
functions in the last section. For example,
\begin{align*}
\sum_{n=0}^{\left\lfloor \sqrt{N/2}\right\rfloor }\sum_{j=-n}^{n+1}(-1)^{n}\bigg(\overline{p}\bigg(\frac{N-3n-j^{2}}{2}-n^{2}\bigg)-\overline{p}\bigg(\frac{N-5n-j^{2}-2}{2}-n^{2}\bigg)\bigg) & \geq0,\\
\sum_{n=0}^{\left\lfloor \sqrt{N/2}\right\rfloor }\sum_{j=-n}^{n}(-1)^{j}\mathrm{(pod}(N-2n^{2}-n-j^{2})-\mathrm{pod}(N-2n^{2}-3n-j^{2}-1)) & \geq0,
\end{align*}
where $\overline{p}(n)$ and $\mathrm{pod}(n)$ denote the number
of overpartitions of $n$ and the number of partitions of $n$ without
repeated odd parts respectively.

It should be mentioned that the core tools for our proofs are these
formulas from Liu's work \cite{L6}.

\section{\label{sec:Pre} Preliminaries}

In this section we collect some useful formulas or identities on basic
hypergeometric series.

Throughout this paper we adopt the standard $q$-notations: the $q$-shifted
factorials are defined as
\[
(a;q)_{0}:=1,\;(a;q)_{n}:=\prod_{k=0}^{n-1}(1-aq^{k}),\;(a;q)_{\infty}:=\prod_{k=0}^{\infty}(1-aq^{k}).
\]
The following compact notation for multiple $q$-shifted factorials
is also used:
\[
(a_{1},a_{2},\cdots,a_{m};q)_{n}:=(a_{1};q)_{n}(a_{2};q)_{n}\cdots(a_{m};q)_{n},
\]
where $n$ is an integer or $\infty$.

The basic hypergeometric series $_{r}\phi_{s}$ is defined by \cite[(1.2.22)]{GR}
\[
_{r}\phi_{s}\left(\begin{matrix}a_{1},a_{2},\cdots,a_{r}\\
b_{1},b_{2},\cdots,b_{s}
\end{matrix};q,z\right):=\sum_{n=0}^{\infty}\frac{(a_{1},a_{2},\cdots,a_{r};q)_{n}}{(q,b_{1},b_{2},\cdots,b_{s};q)_{n}}((-1)^{n}q^{n(n-1)/2})^{1+s-r}z^{n}.
\]

Liu \cite[Theorem 4.3]{L6} established the following $q$-transformation
formula by using a general $q$-transformation formula: for $\max\{|\alpha a|,|\alpha b|,|\alpha abz/q|\}<1,$
we have

\begin{equation}
\begin{aligned} & \frac{(\alpha q,\alpha ab/q;q)_{\infty}}{(\alpha a,\alpha b;q)_{\infty}}\text{}_{4}\phi_{3}\left(\begin{matrix}q/a,q/b,\beta,\gamma\\
c,d,h
\end{matrix};q,\frac{\alpha abz}{q}\right)\\
 & =\sum_{n=0}^{\infty}\frac{(1-\alpha q^{2n})(\alpha,q/a,q/b;q)_{n}(-\alpha ab/q)^{n}q^{{n \choose 2}}}{(1-\alpha)(q,\alpha a,\alpha b;q)_{n}}\text{}_{4}\phi_{3}\left(\begin{matrix}q^{-n},\alpha q^{n},\beta,\gamma\\
c,d,h
\end{matrix};q,qz\right).
\end{aligned}
\label{eq:2-1}
\end{equation}

If $\gamma=h,$ then \eqref{eq:2-1} becomes the following formula:
\begin{equation}
\begin{aligned} & \frac{(\alpha q,\alpha ab/q;q)_{\infty}}{(\alpha a,\alpha b;q)_{\infty}}\text{}_{3}\phi_{2}\left(\begin{matrix}q/a,q/b,\beta\\
c,d
\end{matrix};q,\frac{\alpha abz}{q}\right)\\
 & =\sum_{n=0}^{\infty}\frac{(1-\alpha q^{2n})(\alpha,q/a,q/b;q)_{n}(-\alpha ab/q)^{n}q^{{n \choose 2}}}{(1-\alpha)(q,\alpha a,\alpha b;q)_{n}}\text{}_{3}\phi_{2}\left(\begin{matrix}q^{-n},\alpha q^{n},\beta\\
c,d
\end{matrix};q,qz\right),
\end{aligned}
\label{eq:t2-4}
\end{equation}
where $\max\{|\alpha a|,|\alpha b|,|\alpha abz/q|\}<1.$

When $z=1$ the formula \eqref{eq:t2-4} reduces to the following
identity:

\begin{equation}
\begin{aligned} & \frac{(\alpha q,\alpha ab/q;q)_{\infty}}{(\alpha a,\alpha b;q)_{\infty}}\text{}_{3}\phi_{2}\left(\begin{matrix}q/a,q/b,\beta\\
c,d
\end{matrix};q,\frac{\alpha ab}{q}\right)\\
 & =\sum_{n=0}^{\infty}\frac{(1-\alpha q^{2n})(\alpha,q/a,q/b;q)_{n}(-\alpha ab/q)^{n}q^{{n \choose 2}}}{(1-\alpha)(q,\alpha a,\alpha b;q)_{n}}\text{}_{3}\phi_{2}\left(\begin{matrix}q^{-n},\alpha q^{n},\beta\\
c,d
\end{matrix};q,q\right),
\end{aligned}
\label{eq:1-3}
\end{equation}
where $\max\{|\alpha a|,|\alpha b|,|\alpha ab/q|\}<1.$ This formula
can also be obtained by replacing $(\alpha c,\alpha d)$ by $(c,d)$
in \cite[Theorem 1.7]{L5}.

Setting $\beta=d$ in \eqref{eq:t2-4} we obtain the following $q$-transformation
formula: for $\max\{|\alpha a|,|\alpha b|,|\alpha abz/q|\}<1,$ we
have
\begin{equation}
\begin{aligned} & \frac{(\alpha q,\alpha ab/q;q)_{\infty}}{(\alpha a,\alpha b;q)_{\infty}}\text{}_{2}\phi_{1}\left(\begin{matrix}q/a,q/b\\
c
\end{matrix};q,\frac{\alpha abz}{q}\right)\\
 & =\sum_{n=0}^{\infty}\frac{(1-\alpha q^{2n})(\alpha,q/a,q/b;q)_{n}(-\alpha ab/q)^{n}q^{{n \choose 2}}}{(1-\alpha)(q,\alpha a,\alpha b;q)_{n}}\text{}_{2}\phi_{1}\left(\begin{matrix}q^{-n},\alpha q^{n}\\
c
\end{matrix};q,qz\right).
\end{aligned}
\label{eq:t2-9}
\end{equation}

From \cite[(3.2.5) and  (3.2.6)]{GR} we have
\begin{equation}
\text{}_{3}\phi_{2}\left(\begin{matrix}q^{-n},a,b\\
d,e
\end{matrix};q,\frac{deq^{n}}{ab}\right)=\frac{(e/a;q)_{n}}{(e;q)_{n}}\text{}_{3}\phi_{2}\left(\begin{matrix}q^{-n},a,d/b\\
d,aq^{1-n}/e
\end{matrix};q,q\right),\label{eq:1-1}
\end{equation}
and
\begin{equation}
\begin{aligned} & \text{}_{3}\phi_{2}\left(\begin{matrix}q^{-n},aq^{n},b\\
d,e
\end{matrix};q,\frac{de}{ab}\right)\\
 & =\frac{(aq/d,aq/e;q)_{n}}{(d,e;q)_{n}}\bigg(\frac{de}{aq}\bigg)^{n}\text{}_{3}\phi_{2}\left(\begin{matrix}q^{-n},aq^{n},abq/de\\
aq/d,aq/e
\end{matrix};q,\frac{q}{b}\right).
\end{aligned}
\label{eq:7-1}
\end{equation}

From \cite[Propositions 2.4 and 2.5]{L6}\footnote{The factor $(-1)^{n}$ is missing on the left-hand side of \cite[Propositions 2.4 and 2.5]{L6}.}
we find that

\begin{equation}
\begin{aligned} & (-1)^{n}\frac{(\alpha q;q)_{n}}{(q;q)_{n}}q^{{n+1 \choose 2}}\text{}_{3}\phi_{2}\left(\begin{matrix}q^{-n},\alpha q^{n+1},\alpha cd/q\\
\alpha c,\alpha d
\end{matrix};q,1\right)\\
 & =\sum_{j=0}^{n}(-1)^{j}\frac{(1-\alpha q^{2j})(\alpha,q/c,q/d;q)_{j}}{(1-\alpha)(q,\alpha c,\alpha d;q)_{j}}q^{j(j-3)/2}(\alpha cd)^{j}
\end{aligned}
\label{eq:1-2}
\end{equation}
and
\begin{equation}
\begin{aligned} & (-1)^{n}\frac{(\alpha q;q)_{n}}{(q;q)_{n}}q^{{n+1 \choose 2}}\text{}_{2}\phi_{1}\left(\begin{matrix}q^{-n},\alpha q^{n+1}\\
\alpha c
\end{matrix};q,1\right)\\
 & =\sum_{j=0}^{n}\frac{(1-\alpha q^{2j})(\alpha,q/c;q)_{j}}{(1-\alpha)(q,\alpha c;q)_{j}}q^{j^{2}-j}(\alpha c)^{j}.
\end{aligned}
\label{eq:7-12}
\end{equation}

\section{\label{sec:3} $(a,b)$-Parameterized identities and Hecke-type identities
associated with definite quadratic forms}

In this section we establish many Hecke-type identities associated
with definite quadratic forms.
\begin{thm}
\label{t} For $\max\{|aq^{2}|,|bq^{2}|,|ab|\}<1,$ we have
\begin{equation}
\begin{aligned} & \frac{(q^{2},ab;q^{2})_{\infty}}{(aq^{2},bq^{2};q^{2})_{\infty}}\text{}_{3}\phi_{2}\left(\begin{matrix}q^{2}/a,q^{2}/b,-1\\
q,-q^{2}
\end{matrix};q^{2},ab\right)\\
 & =\sum_{n=0}^{\infty}\sum_{j=-n}^{n}\frac{(1-q^{4n+2})(q^{2}/a,q^{2}/b;q^{2})_{n}(ab)^{n}}{(aq^{2},bq^{2};q^{2})_{n}}q^{n(n-1)+j^{2}}.
\end{aligned}
\label{eq:7-15}
\end{equation}
\end{thm}
\noindent{\it Proof.} Replacing $(q,a,b,d,e)$ by $(q^{2},q^{2n+2},-q,q,-q^{2})$
in \eqref{eq:1-1} we have
\[
\text{}_{3}\phi_{2}\left(\begin{matrix}q^{-2n},q^{2n+2},-q\\
q,-q^{2}
\end{matrix};q^{2},1\right)=\frac{(-q^{-2n};q^{2})_{n}}{(-q^{2};q^{2})_{n}}\text{}_{3}\phi_{2}\left(\begin{matrix}q^{-2n},q^{2n+2},-1\\
q,-q^{2}
\end{matrix};q^{2},q^{2}\right),
\]
which combines the identity $(-q^{-2n};q^{2})_{n}=q^{-n(n+1)}(-q^{2};q^{2})_{n}$
to arrive at
\begin{equation}
\text{}_{3}\phi_{2}\left(\begin{matrix}q^{-2n},q^{2n+2},-1\\
q,-q^{2}
\end{matrix};q^{2},q^{2}\right)=q^{n(n+1)}\text{}_{3}\phi_{2}\left(\begin{matrix}q^{-2n},q^{2n+2},-q\\
q,-q^{2}
\end{matrix};q^{2},1\right).\label{eq:1-7}
\end{equation}
Replcaing $q$ by $q^{2}$ and then setting $\alpha=1,c=q,d=-q^{2}$
in \eqref{eq:1-2} we get
\begin{equation}
(-1)^{n}q^{n(n+1)}\text{}_{3}\phi_{2}\left(\begin{matrix}q^{-2n},q^{2n+2},-q\\
q,-q^{2}
\end{matrix};q^{2},1\right)=\sum_{j=-n}^{n}q^{j^{2}}.\label{eq:1-6}
\end{equation}
Substituting \eqref{eq:1-6} into the right-hand side of \eqref{eq:1-7}
we conclude that
\begin{equation}
\text{}_{3}\phi_{2}\left(\begin{matrix}q^{-2n},q^{2n+2},-1\\
q,-q^{2}
\end{matrix};q^{2},q^{2}\right)=(-1)^{n}\sum_{j=-n}^{n}q^{j^{2}}.\label{eq:7-14}
\end{equation}
We replace $q$ by $q^{2}$ and take $\alpha=q^{2},\beta=-1,c=q,d=-q^{2}$
in \eqref{eq:1-3} to obtain
\begin{equation}
\begin{aligned} & \frac{(q^{2},ab;q^{2})_{\infty}}{(aq^{2},bq^{2};q^{2})_{\infty}}\text{}_{3}\phi_{2}\left(\begin{matrix}q^{2}/a,q^{2}/b,-1\\
q,-q^{2}
\end{matrix};q^{2},ab\right)\\
 & =\sum_{n=0}^{\infty}\frac{(1-q^{4n+2})(q^{2}/a,q^{2}/b;q^{2})_{n}(-ab)^{n}q^{n(n-1)}}{(aq^{2},bq^{2};q^{2})_{n}}\\
 & \;\;\;\;\;\times\text{}_{3}\phi_{2}\left(\begin{matrix}q^{-2n},q^{2n+2},-1\\
q,-q^{2}
\end{matrix};q^{2},q^{2}\right).
\end{aligned}
\label{eq:1-8}
\end{equation}
Then the identity \eqref{eq:7-15} follows readily by substituting
\eqref{eq:7-14} into the right-hand side of \eqref{eq:1-8}. \qed

We now use \eqref{eq:7-15} to deduce the following double series
identity.
\begin{thm}
\label{t1} We have
\[
\sum_{n=1}^{\infty}\frac{q^{n}}{1+q^{2n}}=\sum_{n=1}^{\infty}\sum_{j=-n}^{n}q^{n^{2}+j^{2}}-\sum_{n=1}^{\infty}q^{2n^{2}}.
\]
\end{thm}
\noindent{\it Proof.} Setting $a=1$ and $b=q$ in \eqref{eq:7-15}
and simplifying we find that
\begin{align*}
1+2\sum_{n=1}^{\infty}\frac{q^{n}}{1+q^{2n}} & =\text{}_{3}\phi_{2}\left(\begin{matrix}q^{2},q,-1\\
q,-q^{2}
\end{matrix};q^{2},q\right)\\
 & =\sum_{n=0}^{\infty}(1+q^{2n+1})q^{n^{2}}\sum_{j=-n}^{n}q^{j^{2}}.
\end{align*}
Thus
\begin{align*}
1+2\sum_{n=1}^{\infty}\frac{q^{n}}{1+q^{2n}} & =\sum_{n=0}^{\infty}q^{n^{2}}\sum_{j=-n}^{n}q^{j^{2}}+\sum_{n=0}^{\infty}q^{(n+1)^{2}}\sum_{j=-n}^{n}q^{j^{2}}\\
 & =1+\sum_{n=1}^{\infty}q^{n^{2}}\sum_{j=-n}^{n}q^{j^{2}}+\sum_{n=1}^{\infty}q^{n^{2}}\sum_{j=-n+1}^{n-1}q^{j^{2}}\\
 & =1+2\sum_{n=1}^{\infty}\sum_{j=-n+1}^{n}q^{n^{2}+j^{2}}.
\end{align*}
This means that
\[
\sum_{n=1}^{\infty}\frac{q^{n}}{1+q^{2n}}=\sum_{n=1}^{\infty}\sum_{j=-n+1}^{n}q^{n^{2}+j^{2}}.
\]
This finishes the proof of Theorem \ref{t1}. \qed

From Theorem \ref{t1} we can derive an interesting identity, which
is given in the following corollary.
\begin{cor}
\label{c1} We have
\[
\sum_{n=1}^{\infty}\frac{q^{n}}{1+q^{2n}}=-\frac{1}{4}+\frac{1}{4}\bigg(\sum_{n=-\infty}^{\infty}q^{2n^{2}}\bigg)^{2}+q\bigg(\sum_{n=0}^{\infty}q^{2n(n+1)}\bigg)^{2}.
\]
\end{cor}
From this result we have
\[
1+4\sum_{n=1}^{\infty}\frac{q^{n}}{1+q^{2n}}=\bigg(\sum_{n=-\infty}^{\infty}q^{2n^{2}}\bigg)^{2}+4q\bigg(\sum_{n=0}^{\infty}q^{2n(n+1)}\bigg)^{2},
\]
which, together with \cite[(2)]{ALL}, gives 
\[
\bigg(\sum_{n=-\infty}^{\infty}q^{n^{2}}\bigg)^{2}=\bigg(\sum_{n=-\infty}^{\infty}q^{2n^{2}}\bigg)^{2}+4q\bigg(\sum_{n=0}^{\infty}q^{2n(n+1)}\bigg)^{2}.
\]
This identity can also be deduced by using the results \cite[Theorem 5.4.1 (i),(iv) and Theorem 5.4.2 (iv)]{B}
or \cite[p. 40, Entry 25 (v), (vi)]{B91}.

Before proving Corollary \ref{c1} we need an auxiliary result.
\begin{lem}
\label{l1} We have

\[
\sum_{n=1}^{\infty}\sum_{m=-n+1}^{n}q^{n^{2}+m^{2}}=-\frac{1}{4}+\frac{1}{4}\bigg(\sum_{n=-\infty}^{\infty}q^{2n^{2}}\bigg)^{2}+q\bigg(\sum_{n=0}^{\infty}q^{2n(n+1)}\bigg)^{2}.
\]
\end{lem}
\begin{proof}
On one hand, since
\begin{align*}
\bigg(\sum_{n=-\infty}^{\infty}q^{2n^{2}}\bigg)^{2} & =\bigg(1+2\sum_{n=1}^{\infty}q^{2n^{2}}\bigg)^{2}\\
 & =1+4\sum_{n=1}^{\infty}q^{2n^{2}}+4\bigg(\sum_{n=1}^{\infty}q^{2n^{2}}\bigg)^{2}\\
 & =1+4\bigg(\sum_{n=1}^{\infty}q^{2n^{2}}+\sum_{n=1}^{\infty}\sum_{m=1}^{\infty}q^{2n^{2}+2m^{2}}\bigg),
\end{align*}
we have
\begin{align*}
-\frac{1}{4}+\frac{1}{4}\bigg(\sum_{n=-\infty}^{\infty}q^{2n^{2}}\bigg)^{2} & =\sum_{n=1}^{\infty}\sum_{m=0}^{\infty}q^{2n^{2}+2m^{2}}\\
 & =\sum_{n=1}^{\infty}\sum_{m=0}^{\infty}q^{(n+m)^{2}+(n-m)^{2}}.
\end{align*}
Set $n'=n+m,m'=n-m.$ Then the set of integer pairs:
\[
\{(n,m)\in\mathbb{Z}^{2}|n\geq1,m\geq0\}
\]
 is transformed into the set 
\[
\{(n',m')\in\mathbb{Z}^{2}|n'\geq1,2-n'\leq m'\leq n',n'\equiv m'(\bmod2)\}.
\]
Thus
\begin{equation}
-\frac{1}{4}+\frac{1}{4}\bigg(\sum_{n=-\infty}^{\infty}q^{2n^{2}}\bigg)^{2}=\sum_{n=1}^{\infty}\sum_{\substack{1-n\leq m\leq n\\
m\equiv n(\bmod2)
}
}q^{n^{2}+m^{2}}.\label{eq:1-4}
\end{equation}
On the other hand,
\begin{align*}
q\bigg(\sum_{n=0}^{\infty}q^{2n(n+1)}\bigg)^{2} & =\sum_{n=0}^{\infty}\sum_{m=0}^{\infty}q^{2n(n+1)+2m(m+1)+1}\\
 & =\sum_{n=0}^{\infty}\sum_{m=0}^{\infty}q^{(n+m+1)^{2}+(n-m)^{2}}.
\end{align*}
Take $n'=n+m+1,m'=n-m.$ Then the set of integer pairs:
\[
\{(n,m)\in\mathbb{Z}^{2}|n\geq0,m\geq0\}
\]
 is transformed into the set 
\[
\{(n',m')\in\mathbb{Z}^{2}|n'\geq1,1-n'\leq m'\leq n'-1,n'+1\equiv m'(\bmod2)\}.
\]
Therefore
\begin{equation}
q\bigg(\sum_{n=0}^{\infty}q^{2n(n+1)}\bigg)^{2}=\sum_{n=1}^{\infty}\sum_{\substack{1-n\leq m\leq n\\
m\equiv n+1(\bmod2)
}
}q^{n^{2}+m^{2}}\label{eq:1-5}
\end{equation}

Combining the identities \eqref{eq:1-4} and \eqref{eq:1-5} we can
easily obtain the result. This completes the proof of Lemma \ref{l1}.
\end{proof}
\noindent{\it Proof of Corollary \ref{c1}.} The identity in Corollary
\ref{c1} can be obtained by combining Theorem \ref{t1} and Lemma
\ref{l1}. \qed

We next employ \eqref{eq:7-15} to deduce another double series identity,
which is given in the following theorem.
\begin{thm}
\label{t3-5} We have
\[
\sum_{n=1}^{\infty}\frac{(-1)^{n}q^{n^{2}+n}}{(1+q^{2n})(q;q^{2})_{n}}=\sum_{n=1}^{\infty}\sum_{j=-n}^{n}(-1)^{n}q^{2n^{2}+j^{2}}-\sum_{n=1}^{\infty}(-1)^{n}q^{3n^{2}}.
\]
\end{thm}
\begin{proof}
It is easily seen that 
\[
1+2\sum_{n=1}^{\infty}\frac{(-1)^{n}q^{n^{2}+n}}{(1+q^{2n})(q;q^{2})_{n}}=\lim_{b\rightarrow0}\text{}_{3}\phi_{2}\left(\begin{matrix}q^{2},q^{2}/b,-1\\
q,-q^{2}
\end{matrix};q^{2},b\right).
\]
Set $a=1,b\rightarrow0$ in \eqref{eq:7-15}. We obtain
\[
1+2\sum_{n=1}^{\infty}\frac{(-1)^{n}q^{n^{2}+n}}{(1+q^{2n})(q;q^{2})_{n}}=\sum_{n=0}^{\infty}\sum_{j=-n}^{n}(-1)^{n}(1-q^{4n+2})q^{2n^{2}+j^{2}}.
\]
Since
\begin{align*}
 & \sum_{n=0}^{\infty}\sum_{j=-n}^{n}(-1)^{n}(1-q^{4n+2})q^{2n^{2}+j^{2}}\\
 & =\sum_{n=0}^{\infty}\sum_{j=-n}^{n}(-1)^{n}q^{2n^{2}+j^{2}}+\sum_{n=0}^{\infty}\sum_{j=-n}^{n}(-1)^{n+1}q^{2(n+1)^{2}+j^{2}}\\
 & =1+\sum_{n=1}^{\infty}\sum_{j=-n}^{n}(-1)^{n}q^{2n^{2}+j^{2}}+\sum_{n=1}^{\infty}\sum_{j=-n+1}^{n-1}(-1)^{n}q^{2n^{2}+j^{2}}\\
 & =1+2\sum_{n=1}^{\infty}\sum_{j=-n+1}^{n}(-1)^{n}q^{2n^{2}+j^{2}},
\end{align*}
we see that
\[
\sum_{n=1}^{\infty}\frac{(-1)^{n}q^{n^{2}+n}}{(1+q^{2n})(q;q^{2})_{n}}=\sum_{n=1}^{\infty}\sum_{j=-n+1}^{n}(-1)^{n}q^{2n^{2}+j^{2}}.
\]
This completes the proof of Theorem \ref{t3-5}.
\end{proof}
Substituting $a=1,b=-q$ into \eqref{eq:7-15} and then replacing
$q$ by $-q$ in the resulting identity we can easily obtain the formula
\eqref{eq:6-1}.

From the formula \eqref{eq:7-15} we can deduce many other Hecke-type
identities associated with definite quadratic forms. Set $(a,b)=(q,-q),(-1,\pm q),(0,\pm q),(0,-1)$
and $(0,0)$ in \eqref{eq:7-15}. We get
\begin{align}
\sum_{n=0}^{\infty}\frac{(-1)^{n}(-q;q^{2})_{n}q^{2n}}{(1+q^{2n})(q^{2};q^{2})_{n}} & =\frac{(q^{2};q^{4})_{\infty}}{2(q^{4};q^{4})_{\infty}}\sum_{n=0}^{\infty}\sum_{j=-n}^{n}(-1)^{n}q^{n(n+1)+j^{2}},\nonumber \\
\sum_{n=0}^{\infty}\frac{(-1;q^{2})_{n}q^{n}}{(q^{2};q^{2})_{n}} & =\frac{(-q;q)_{\infty}}{(q;q)_{\infty}}\sum_{n=0}^{\infty}\sum_{j=-n}^{n}(-1)^{j}(1-q^{2n+1})q^{n^{2}+j^{2}},\nonumber \\
\sum_{n=0}^{\infty}\frac{(-1;q)_{2n}q^{n}}{(q;q)_{2n}} & =\frac{(-q;q)_{\infty}}{(q;q)_{\infty}}\sum_{n=0}^{\infty}\sum_{j=-n}^{n}(1-q^{2n+1})q^{n^{2}+j^{2}},\nonumber \\
\sum_{n=0}^{\infty}\frac{(-1)^{n}q^{n^{2}+2n}}{(1+q^{2n})(q^{2};q^{2})_{n}} & =\frac{(q;q^{2})_{\infty}}{2(q^{2};q^{2})_{\infty}}\sum_{n=0}^{\infty}\sum_{j=-n}^{n}(-1)^{n}(1+q^{2n+1})q^{2n^{2}+n+j^{2}},\label{eq:5-8}\\
\sum_{n=0}^{\infty}\frac{(-q;q^{2})_{n}q^{n^{2}+2n}}{(1+q^{2n})(q;q)_{2n}} & =\frac{(-q;q^{2})_{\infty}}{2(q^{2};q^{2})_{\infty}}\sum_{n=0}^{\infty}\sum_{j=-n}^{n}(1-q^{2n+1})q^{2n^{2}+n+j^{2}},\nonumber \\
\sum_{n=0}^{\infty}\frac{(-1;q^{2})_{n}}{(q;q)_{2n}}q^{n^{2}+n} & =\frac{(-q^{2};q^{2})_{\infty}}{(q^{2};q^{2})_{\infty}}\sum_{n=0}^{\infty}\sum_{j=-n}^{n}(1-q^{4n+2})q^{2n^{2}+j^{2}},\nonumber \\
\sum_{n=0}^{\infty}\frac{q^{2n^{2}+2n}}{(1+q^{2n})(q;q)_{2n}} & =\frac{1}{2(q^{2};q^{2})_{\infty}}\sum_{n=0}^{\infty}\sum_{j=-n}^{n}(1-q^{4n+2})q^{3n^{2}+n+j^{2}}.\nonumber 
\end{align}

\begin{thm}
\label{t3-3} For $\max\{|aq^{2}|,|bq^{2}|,|ab|\}<1,$ we have
\begin{equation}
\begin{aligned} & \frac{(q^{2},ab;q^{2})_{\infty}}{(aq^{2},bq^{2};q^{2})_{\infty}}\text{}_{3}\phi_{2}\left(\begin{matrix}q^{2}/a,q^{2}/b,q\\
-q^{2},q^{3}
\end{matrix};q^{2},ab\right)\\
 & =(1-q)\sum_{n=0}^{\infty}\sum_{j=-n}^{n}\frac{(1+q^{2n+1})(q^{2}/a,q^{2}/b;q^{2})_{n}(-ab)^{n}}{(aq^{2},bq^{2};q^{2})_{n}}q^{n^{2}+j^{2}}.
\end{aligned}
\label{eq:7-13}
\end{equation}
\end{thm}
In order to derive \eqref{eq:7-13} we first prove two important results.
\begin{lem}
For any non-negative integer $n,$ we have
\begin{equation}
_{3}\phi_{2}\left(\begin{matrix}q^{-2n},q^{2n+2},q^{2}\\
-q^{2},-q^{3}
\end{matrix};q^{2},q\right)=\frac{(-1)^{n}(1+q)}{q^{n^{2}}(1+q^{2n+1})}\sum_{j=-n}^{n}(-1)^{j}q^{j^{2}}.\label{eq:7-4}
\end{equation}
\end{lem}
\noindent{\it Proof.} Replacing $q$ by $-q$ in \eqref{eq:1-6}
yields
\begin{equation}
(-1)^{n}q^{n(n+1)}\text{}_{3}\phi_{2}\left(\begin{matrix}q^{-2n},q^{2n+2},q\\
-q,-q^{2}
\end{matrix};q^{2},1\right)=\sum_{j=-n}^{n}(-1)^{j}q^{j^{2}}.\label{eq:7-2}
\end{equation}
Replacing $q$ by $q^{2}$ and then taking $a=b=q^{2},d=-q^{2},e=-q^{3}$
in \eqref{eq:7-1} we conclude that
\[
_{3}\phi_{2}\left(\begin{matrix}q^{-2n},q^{2n+2},q^{2}\\
-q^{2},-q^{3}
\end{matrix};q^{2},q\right)=\frac{(1+q)q^{n}}{1+q^{2n+1}}{}_{3}\phi_{2}\left(\begin{matrix}q^{-2n},q^{2n+2},q\\
-q,-q^{2}
\end{matrix};q^{2},1\right).
\]
Then the result \eqref{eq:7-4} follows easily by substituting \eqref{eq:7-2}
into the right-hand side of the above identity. \qed
\begin{lem}
For any nonnegative integer $n$, we have
\begin{equation}
_{3}\phi_{2}\left(\begin{matrix}q^{-2n},q^{2n+2},q\\
-q^{2},q^{3}
\end{matrix};q^{2},q^{2}\right)=\frac{(1-q)q^{n}}{1-q^{2n+1}}\sum_{j=-n}^{n}q^{j^{2}}.\label{eq:7-7}
\end{equation}
\end{lem}
\begin{proof}
Replacing $q$ by $-q$ in \eqref{eq:7-4} gives
\begin{equation}
_{3}\phi_{2}\left(\begin{matrix}q^{-2n},q^{2n+2},q^{2}\\
-q^{2},q^{3}
\end{matrix};q^{2},-q\right)=\frac{1-q}{q^{n^{2}}(1-q^{2n+1})}\sum_{j=-n}^{n}q^{j^{2}}.\label{eq:7-6}
\end{equation}
Replacing $q$ by $q^{2}$ and then putting $a=q^{2n+2},b=q^{2},d=q^{3},e=-q^{2}$
in \eqref{eq:1-1} we find
\[
\text{}_{3}\phi_{2}\left(\begin{matrix}q^{-2n},q^{2n+2},q\\
-q^{2},q^{3}
\end{matrix};q^{2},q^{2}\right)=\frac{(-q^{2};q^{2})_{n}}{(-q^{-2n};q^{2})_{n}}\text{}_{3}\phi_{2}\left(\begin{matrix}q^{-2n},q^{2n+2},q^{2}\\
-q^{2},q^{3}
\end{matrix};q^{2},-q\right).
\]
Then the identity \eqref{eq:7-7} can be obtained by substituting
\eqref{eq:7-6} into the right-hand side of the above formula and
then applying the identity $\frac{(-q^{2};q^{2})_{n}}{(-q^{-2n};q^{2})_{n}}=q^{n^{2}+n}.$
\end{proof}
We are now in the position to prove Theorem \ref{t3-3}.

\noindent{\it Proof of Theorem \ref{t3-3}.} Replace $q$ by $q^{2}$
and then set $\alpha=q^{2},\beta=q,c=-q^{2},d=q^{3}$ in \eqref{eq:1-3}.
We obtain
\begin{align*}
 & \frac{(q^{2},ab;q^{2})_{\infty}}{(aq^{2},bq^{2};q^{2})_{\infty}}\text{}_{3}\phi_{2}\left(\begin{matrix}q^{2}/a,q^{2}/b,q\\
-q^{2},q^{3}
\end{matrix};q^{2},ab\right)\\
 & =\sum_{n=0}^{\infty}\frac{(1-q^{4n+2})(q^{2}/a,q^{2}/b;q^{2})_{n}(-ab)^{n}q^{n(n-1)}}{(aq^{2},bq^{2};q^{2})_{n}}\text{}_{3}\phi_{2}\left(\begin{matrix}q^{-2n},q^{2n+2},q\\
-q^{2},q^{3}
\end{matrix};q^{2},q^{2}\right).
\end{align*}
Then the result \eqref{eq:7-13} follows easily by substituting \eqref{eq:7-7}
into the right-hand side of the above equation. \qed

From Theorem \ref{t3-3} we can derive some Hecke-type identities
associated with definite quadratic forms. Taking $(a,b)=(1,-q)$ and
$(-1,-q)$ in \eqref{eq:7-13} we get
\[
\sum_{n=0}^{\infty}\frac{(-q;q^{2})_{n}(-q)^{n}}{(-q^{2};q^{2})_{n}(1-q^{2n+1})}=\sum_{n=0}^{\infty}\sum_{j=-n}^{n}q^{n^{2}+n+j^{2}},
\]
and

\begin{equation}
\sum_{n=0}^{\infty}\frac{(-q;q^{2})_{n}q{}^{n}}{(q^{2};q^{2})_{n}(1-q^{2n+1})}=\frac{(-q;q)_{\infty}}{(q;q)_{\infty}}\sum_{n=0}^{\infty}\sum_{j=-n}^{n}(-1)^{n}q^{n^{2}+n+j^{2}}.\label{eq:5-1}
\end{equation}

Let $(a,b)=(1,0)$, $(-q,0)$, $(-1,0)$ and $(0,0)$ in \eqref{eq:7-13}.
We obtain
\begin{align}
\sum_{n=0}^{\infty}\frac{(-1)^{n}q^{n^{2}+n}}{(-q^{2};q^{2})_{n}(1-q^{2n+1})} & =\sum_{n=0}^{\infty}\sum_{j=-n}^{n}(1+q^{2n+1})q^{2n^{2}+n+j^{2}},\nonumber \\
\sum_{n=0}^{\infty}\frac{(-q;q^{2})_{n}q^{n^{2}+2n}}{(1-q^{2n+1})(q^{4};q^{4})_{n}} & =\frac{(-q;q^{2})_{\infty}}{(q^{2};q^{2})_{\infty}}\sum_{n=0}^{\infty}\sum_{j=-n}^{n}(-1)^{n}q^{2n^{2}+2n+j^{2}},\label{eq:5-2}\\
\sum_{n=0}^{\infty}\frac{q^{n^{2}+n}}{(1-q^{2n+1})(q^{2};q^{2})_{n}} & =\frac{(-q^{2};q^{2})_{\infty}}{(q^{2};q^{2})_{\infty}}\sum_{n=0}^{\infty}\sum_{j=-n}^{n}(-1)^{n}(1+q^{2n+1})q^{2n^{2}+n+j^{2}},\label{eq:5-3}
\end{align}
and 
\begin{equation}
\sum_{n=0}^{\infty}\frac{q^{2n^{2}+2n}}{(1-q^{2n+1})(q^{4};q^{4})_{n}}=\frac{1}{(q^{2};q^{2})_{\infty}}\sum_{n=0}^{\infty}\sum_{j=-n}^{n}(-1)^{n}(1+q^{2n+1})q^{3n^{2}+2n+j^{2}}.\label{eq:5-4}
\end{equation}

\begin{thm}
For $\max\{|aq^{2}|,|bq^{2}|,|ab/q^{2}|\}<1,$ we have
\begin{equation}
\begin{aligned} & \frac{(q^{2},ab;q^{2})_{\infty}}{(aq^{2},bq^{2};q^{2})_{\infty}}\text{}_{3}\phi_{2}\left(\begin{matrix}q^{2}/a,q^{2}/b,-q\\
q,-q^{2}
\end{matrix};q^{2},\frac{ab}{q^{2}}\right)\\
 & =\sum_{n=0}^{\infty}\sum_{j=-n}^{n}\frac{(1-q^{4n+2})(q^{2}/a,q^{2}/b;q^{2})_{n}(ab/q^{2})^{n}q^{j^{2}}}{(aq^{2},bq^{2};q^{2})_{n}}.
\end{aligned}
\label{eq:2-2}
\end{equation}
\end{thm}
\noindent{\it Proof.} Replace $q$ by $q^{2}$ in \eqref{eq:t2-4}
and then take $z=q^{-2},\alpha=q^{2},\beta=q,c=-q,d=-q^{2}$ in the
resulting identity. We obtain
\begin{align*}
 & \frac{(q^{2},ab;q^{2})_{\infty}}{(aq^{2},bq^{2};q^{2})_{\infty}}\text{}_{3}\phi_{2}\left(\begin{matrix}q^{2}/a,q^{2}/b,-q\\
q,-q^{2}
\end{matrix};q^{2},\frac{ab}{q^{2}}\right)\\
 & =\sum_{n=0}^{\infty}\frac{(1-q^{4n+2})(q^{2}/a,q^{2}/b;q^{2})_{n}(-ab)^{n}q^{n(n-1)}}{(aq^{2},bq^{2};q^{2})_{n}}\text{}_{3}\phi_{2}\left(\begin{matrix}q^{-2n},q^{2n+2},-q\\
q,-q^{2}
\end{matrix};q^{2},1\right).
\end{align*}
Then the result \eqref{eq:2-2} follows readily by substituting \eqref{eq:1-6}
into the right-hand side of this identity. \qed

From the identity \eqref{eq:2-2} we can obtain certain Hecke-type
identities associated with definite quadratic forms. Taking $(a,b)=(0,0),(-1,0),(1,0),(q,0)$
and $(-q,0)$ in \eqref{eq:2-2} gives
\begin{align*}
\sum_{n=0}^{\infty}\frac{(-q;q^{2})_{n}q^{2n^{2}}}{(q;q^{2})_{n}(q^{4};q^{4})_{n}} & =\frac{1}{(q^{2};q^{2})_{\infty}}\sum_{n=0}^{\infty}\sum_{j=-n}^{n}(1-q^{4n+2})q^{2n^{2}+j^{2}},\\
\sum_{n=0}^{\infty}\frac{(-q;q^{2})_{n}q^{n^{2}-n}}{(q;q)_{2n}} & =\frac{(-q^{2};q^{2})_{\infty}}{(q^{2};q^{2})_{\infty}}\sum_{n=0}^{\infty}\sum_{j=-n}^{n}(1-q^{4n+2})q^{n^{2}-n+j^{2}},\\
\sum_{n=0}^{\infty}\frac{(-1)^{n}(-q;q^{2})_{n}q^{n^{2}-n}}{(q,-q^{2};q^{2})_{n}} & =\sum_{n=0}^{\infty}\sum_{j=-n}^{n}(-1)^{n}(1-q^{4n+2})q^{n^{2}-n+j^{2}},\\
\sum_{n=0}^{\infty}\frac{(-1)^{n}(-q;q^{2})_{n}q^{n^{2}}}{(q^{4};q^{4})_{n}} & =\frac{(q;q^{2})_{\infty}}{(q^{2};q^{2})_{\infty}}\sum_{n=0}^{\infty}\sum_{j=-n}^{n}(-1)^{n}(1+q^{2n+1})q^{n^{2}+j^{2}},
\end{align*}
and
\[
\sum_{n=0}^{\infty}\frac{(-q;q^{2})_{n}^{2}q^{n^{2}}}{(q;q^{2})_{n}(q^{4};q^{4})_{n}}=\frac{(-q;q^{2})_{\infty}}{(q^{2};q^{2})_{\infty}}\sum_{n=0}^{\infty}\sum_{j=-n}^{n}(1-q^{2n+1})q^{n^{2}+j^{2}}.
\]

\begin{thm}
\label{t3-2} For $\max\{|aq^{2}|,|bq^{2}|,|ab/q|\}<1,$ we have
\begin{equation}
\begin{aligned} & \frac{(q^{2},ab;q^{2})_{\infty}}{(aq^{2},bq^{2};q^{2})_{\infty}}\sum_{n=0}^{\infty}\frac{(q^{2}/a,q^{2}/b;q^{2})_{n}}{(-q;q)_{2n+1}}\bigg(\frac{ab}{q}\bigg)^{n}\\
 & =\sum_{n=0}^{\infty}\sum_{j=-n}^{n}(-1)^{j}\frac{(1-q^{2n+1})(q^{2}/a,q^{2}/b;q^{2})_{n}(ab/q)^{n}q^{j^{2}}}{(aq^{2},bq^{2};q^{2})_{n}}.
\end{aligned}
\label{eq:7-3}
\end{equation}
\end{thm}
\noindent{\it Proof.} Replace $q$ by $q^{2}$ and set $\alpha=\beta=q^{2},c=-q^{2},d=-q^{3},z=q^{-1}$
in \eqref{eq:t2-4}. We arrive at
\begin{align*}
 & \frac{(q^{2},ab;q^{2})_{\infty}}{(aq^{2},bq^{2};q^{2})_{\infty}}\text{}_{3}\phi_{2}\left(\begin{matrix}q^{2}/a,q^{2}/b,q^{2}\\
-q^{2},-q^{3}
\end{matrix};q^{2},\frac{ab}{q}\right)\\
 & =\sum_{n=0}^{\infty}\frac{(1-q^{4n+2})(q^{2}/a,q^{2}/b;q^{2})_{n}(-ab)^{n}q^{n^{2}-n}}{(aq^{2},bq^{2};q^{2})_{n}}\text{}_{3}\phi_{2}\left(\begin{matrix}q^{-2n},q^{2n+2},q^{2}\\
-q^{2},-q^{3}
\end{matrix};q^{2},q\right).
\end{align*}
Then the formula \eqref{eq:7-3} follows immediately by substituting
\eqref{eq:7-4} into the right-hand side of this identity. \qed

Theorem \ref{t3-2} has the following consequence.
\begin{thm}
We have
\begin{equation}
\sum_{n=0}^{\infty}\frac{(-1)^{n}(q^{2};q^{2})_{n}q^{n^{2}}}{(-q;q)_{2n+1}}=1+2\sum_{n=1}^{\infty}\sum_{j=-n+1}^{n}(-1)^{n+j}q^{n^{2}+j^{2}}.\label{eq:8-1}
\end{equation}
\end{thm}
\begin{proof}
It is easy to see that
\begin{align*}
 & \sum_{n=0}^{\infty}\sum_{j=-n}^{n}(-1)^{n+j}(1-q^{2n+1})q^{n^{2}+j^{2}}\\
 & =\sum_{n=0}^{\infty}\sum_{j=-n}^{n}(-1)^{n+j}q^{n^{2}+j^{2}}+\sum_{n=0}^{\infty}\sum_{j=-n}^{n}(-1)^{n+1+j}q^{(n+1)^{2}+j^{2}}\\
 & =1+\sum_{n=1}^{\infty}\sum_{j=-n}^{n}(-1)^{n+j}q^{n^{2}+j^{2}}+\sum_{n=1}^{\infty}\sum_{j=-n+1}^{n-1}(-1)^{n+j}q^{n^{2}+j^{2}}\\
 & =1+2\sum_{n=1}^{\infty}\sum_{j=-n+1}^{n}(-1)^{n+j}q^{n^{2}+j^{2}}.
\end{align*}
Setting $a=1,b\rightarrow0$ in \eqref{eq:7-3} we have
\[
\sum_{n=0}^{\infty}\frac{(-1)^{n}(q^{2};q^{2})_{n}q^{n^{2}}}{(-q;q)_{2n+1}}=\sum_{n=0}^{\infty}\sum_{j=-n}^{n}(-1)^{n+j}(1-q^{2n+1})q^{n^{2}+j^{2}}.
\]
Then the formula \eqref{eq:8-1} follows by combining the above two
identities.
\end{proof}
From Theorem \ref{t3-2} we deduce some other Hecke-type identities
associated with definite quadratic forms. Putting $(a,b)=(0,0),(-1,0)$
and $(q,0)$ in \eqref{eq:7-3} produces 
\begin{align*}
 & \sum_{n=0}^{\infty}\frac{q^{2n^{2}+n}}{(-q;q)_{2n+1}}=\frac{1}{(q^{2};q^{2})_{\infty}}\sum_{n=0}^{\infty}\sum_{j=-n}^{n}(-1)^{j}(1-q^{2n+1})q^{2n^{2}+n+j^{2}},\\
 & \sum_{n=0}^{\infty}\frac{q^{n^{2}}}{(-q;q^{2})_{n+1}}=\frac{(-q^{2};q^{2})_{\infty}}{(q^{2};q^{2})_{\infty}}\sum_{n=0}^{\infty}\sum_{j=-n}^{n}(-1)^{j}(1-q^{2n+1})q^{n^{2}+j^{2}},\\
 & \sum_{n=0}^{\infty}\frac{(-1)^{n}(q;q^{2})_{n}q^{n^{2}+n}}{(-q;q)_{2n+1}}=\frac{(q;q^{2})_{\infty}}{(q^{2};q^{2})_{\infty}}\sum_{n=0}^{\infty}\sum_{j=-n}^{n}(-1)^{n+j}q^{n^{2}+n+j^{2}}.
\end{align*}

\begin{thm}
For $\max\{|aq^{4}|,|bq^{4}|,|ab|\}<1,$ we have
\begin{equation}
\begin{aligned} & \frac{(q^{2},abq^{2};q^{2})_{\infty}}{(aq^{4},bq^{4};q^{2})_{\infty}}\text{}_{3}\phi_{2}\left(\begin{matrix}q^{2}/a,q^{2}/b,-q\\
-q^{2},q^{3}
\end{matrix};q^{2},ab\right)\\
 & =(1-q)\sum_{n=0}^{\infty}\sum_{j=-n}^{n+1}\frac{(1-q^{4n+4})(q^{2}/a,q^{2}/b;q^{2})_{n}(ab)^{n}q^{j^{2}}}{(aq^{4},bq^{4};q^{2})_{n}}.
\end{aligned}
\label{eq:9-2}
\end{equation}
\end{thm}
\begin{proof}
Replace $q$ by $q^{2}$ in \eqref{eq:1-2} and then take $\alpha=q^{2},c=-1,d=q.$
We get
\begin{align*}
 & (-1)^{n}(1-q^{2n+2})q^{n^{2}+n}{}_{3}\phi_{2}\left(\begin{matrix}q^{-2n},q^{2n+4},-q\\
-q^{2},q^{3}
\end{matrix};q^{2},1\right)\\
 & =(1-q)\sum_{j=0}^{n}(1+q^{2j+1})q^{j^{2}}=(1-q)\sum_{j=-n}^{n+1}q^{j^{2}}.
\end{align*}
Then
\begin{equation}
_{3}\phi_{2}\left(\begin{matrix}q^{-2n},q^{2n+4},-q\\
-q^{2},q^{3}
\end{matrix};q^{2},1\right)=\frac{(-1)^{n}(1-q)}{q^{n^{2}+n}(1-q^{2n+2})}\sum_{j=-n}^{n+1}q^{j^{2}}.\label{eq:9-1}
\end{equation}
Replacing $q$ by $q^{2}$ in \eqref{eq:t2-4} and then setting $\alpha=q^{4},\beta=-q,c=-q^{2},d=q^{3},z=q^{-2}$
we find that
\[
\begin{aligned} & \frac{(q^{2},abq^{2};q^{2})_{\infty}}{(aq^{4},bq^{4};q^{2})_{\infty}}\text{}_{3}\phi_{2}\left(\begin{matrix}q^{2}/a,q^{2}/b,-q\\
-q^{2},q^{3}
\end{matrix};q^{2},ab\right)\\
 & =\sum_{n=0}^{\infty}\frac{(1-q^{4n+4})(1-q^{2n+2})(q^{2}/a,q^{2}/b;q^{2})_{n}(-ab)^{n}q^{n(n+1)}}{(aq^{4},bq^{4};q^{2})_{n}}\\
 & \;\times\text{}_{3}\phi_{2}\left(\begin{matrix}q^{-2n},q^{2n+4},-q\\
-q^{2},q^{3}
\end{matrix};q^{2},1\right).
\end{aligned}
\]
Then the identity \eqref{eq:9-2} follows readily by substituting
\eqref{eq:9-1} into the above formula.
\end{proof}
Some Hecke-type identities associated with definite quadratic forms
can be deduced from the identity \eqref{eq:9-2}. Taking $(a,b)=(\pm1,0),(\pm q^{-1},0)$
and $(0,0)$ in \eqref{eq:9-2} we derive
\begin{align*}
\sum_{n=0}^{\infty}\frac{(-1)^{n}(-q;q^{2})_{n}q^{n^{2}+n}}{(-q^{2};q^{2})_{n}(q;q^{2})_{n+1}} & =\sum_{n=0}^{\infty}\sum_{j=-n}^{n+1}(-1)^{n}(1+q^{2n+2})q^{n^{2}+n+j^{2}},\\
\sum_{n=0}^{\infty}\frac{(-q;q^{2})_{n}q^{n^{2}+n}}{(q;q)_{2n+1}} & =\frac{(-q^{2};q^{2})_{\infty}}{(q^{2};q^{2})_{\infty}}\sum_{n=0}^{\infty}\sum_{j=-n}^{n+1}(1-q^{2n+2})q^{n^{2}+n+j^{2}},\\
\sum_{n=0}^{\infty}\frac{(-1)^{n}(-q;q^{2})_{n}q^{n^{2}}}{(q^{4};q^{4})_{n}} & =\frac{(q;q^{2})_{\infty}}{(q^{2};q^{2})_{\infty}}\sum_{n=0}^{\infty}\sum_{j=-n}^{n+1}(-1)^{n}(1-q^{4n+4})q^{n^{2}+j^{2}},\\
\sum_{n=0}^{\infty}\frac{(-q;q^{2})_{n}(-q;q^{2})_{n+1}q^{n^{2}}}{(-q^{2};q^{2})_{n}(q;q)_{2n+1}} & =\frac{(-q;q^{2})_{\infty}}{(q^{2};q^{2})_{\infty}}\sum_{n=0}^{\infty}\sum_{j=-n}^{n+1}(1-q^{4n+4})q^{n^{2}+j^{2}},\\
\sum_{n=0}^{\infty}\frac{(-q;q^{2})_{n}q^{2n^{2}+2n}}{(-q^{2};q^{2})_{n}(q;q)_{2n+1}} & =\frac{1}{(q^{2};q^{2})_{\infty}}\sum_{n=0}^{\infty}\sum_{j=-n}^{n+1}(1-q^{4n+4})q^{2n^{2}+2n+j^{2}}.
\end{align*}

\begin{thm}
\label{t3-9} For $\max\{|aq^{4}|,|bq^{4}|,|abq^{2}|\}<1,$ we have
\begin{equation}
\begin{aligned} & \frac{(q^{2},abq^{2};q^{2})_{\infty}}{(aq^{4},bq^{4};q^{2})_{\infty}}\text{}_{3}\phi_{2}\left(\begin{matrix}q^{2}/a,q^{2}/b,q\\
-q^{2},q^{3}
\end{matrix};q^{2},abq^{2}\right)\\
 & =(1-q)\sum_{n=0}^{\infty}\sum_{j=-n}^{n+1}\frac{(1-q^{4n+4})(q^{2}/a,q^{2}/b;q^{2})_{n}(-ab)^{n}q^{n(n+2)+j^{2}}}{(aq^{4},bq^{4};q^{2})_{n}}.
\end{aligned}
\label{eq:9-4}
\end{equation}
\end{thm}
Before proving Theorem \ref{t3-9} we need one auxiliary result.
\begin{lem}
For any nonnegative integer $n$, we have
\begin{equation}
_{3}\phi_{2}\left(\begin{matrix}q^{-2n},q^{2n+4},q\\
-q^{2},q^{3}
\end{matrix};q^{2},q^{2}\right)=\frac{q^{n}(1-q)}{1-q^{2n+2}}\sum_{j=-n}^{n+1}q^{j^{2}}.\label{eq:9-3}
\end{equation}
\end{lem}
\begin{proof}
Replace $q$ by $q^{2}$ and then set $a=q^{2n+4},b=-q,d=-q^{2},e=q^{3}$
in \eqref{eq:1-1}. We have
\[
_{3}\phi_{2}\left(\begin{matrix}q^{-2n},q^{2n+4},-q\\
-q^{2},q^{3}
\end{matrix};q^{2},1\right)=\frac{(q^{-1-2n};q^{2})_{n}}{(q^{3};q^{2})_{n}}{}_{3}\phi_{2}\left(\begin{matrix}q^{-2n},q^{2n+4},q\\
-q^{2},q^{3}
\end{matrix};q^{2},q^{2}\right).
\]
Then \eqref{eq:9-3} follows by substituting \eqref{eq:9-1} into
the left-hand side of the above identity and applying the equation
$(q^{-1-2n};q^{2})_{n}=(-1)^{n}q^{-n(n+2)}(q^{3};q^{2})_{n}.$
\end{proof}
We are now ready to prove Theorem \ref{t3-9}.

\noindent{\it Proof of Theorem \ref{t3-9}.} Replacing $q$ by $q^{2}$
in \eqref{eq:1-3} and then putting $\alpha=q^{4},\beta=q,c=-q^{2},d=q^{3}$
we get
\[
\begin{aligned} & \frac{(q^{2},abq^{2};q^{2})_{\infty}}{(aq^{4},bq^{4};q^{2})_{\infty}}\text{}_{3}\phi_{2}\left(\begin{matrix}q^{2}/a,q^{2}/b,q\\
-q^{2},q^{3}
\end{matrix};q^{2},abq^{2}\right)\\
 & =\sum_{n=0}^{\infty}\frac{(1-q^{4n+4})(1-q^{2n+2})(q^{2}/a,q^{2}/b;q^{2})_{n}(-ab)^{n}q^{n(n+1)}}{(aq^{4},bq^{4};q^{2})_{n}}\\
 & \;\times\text{}_{3}\phi_{2}\left(\begin{matrix}q^{-2n},q^{2n+4},q\\
-q^{2},q^{3}
\end{matrix};q^{2},q^{2}\right).
\end{aligned}
\]
Then the identity \eqref{eq:9-4} follows quickly by sustituting \eqref{eq:9-3}
into the right-hand side of the above formula. \qed

Some Hecke-type identities associated with definite quadratic forms
can be derived from the identity \eqref{eq:9-4}. Take $(a,b)=(1,-1),(1,\pm q^{-1})$
and $(0,\pm1)$ in \eqref{eq:9-4}. We obtain
\begin{align}
\sum_{n=0}^{\infty}\frac{(-1)^{n}q^{2n}}{1-q^{2n+1}} & =\sum_{n=0}^{\infty}\sum_{j=-n}^{n+1}q^{n^{2}+2n+j^{2}},\nonumber \\
\sum_{n=0}^{\infty}\frac{(q;q^{2})_{n}q^{n}}{(-q^{2};q^{2})_{n}} & =\sum_{n=0}^{\infty}\sum_{j=-n}^{n+1}(-1)^{n}(1+q^{2n+2})q^{n^{2}+n+j^{2}},\nonumber \\
\sum_{n=0}^{\infty}\frac{(-q;q^{2})_{n+1}(-q)^{n}}{(1-q^{2n+1})(-q^{2};q^{2})_{n}} & =\sum_{n=0}^{\infty}\sum_{j=-n}^{n+1}(1+q^{2n+2})q^{n^{2}+n+j^{2}},\nonumber \\
\sum_{n=0}^{\infty}\frac{(-1)^{n}q{}^{n^{2}+3n}}{(1-q^{2n+1})(-q^{2};q^{2})_{n}} & =\sum_{n=0}^{\infty}\sum_{j=-n}^{n+1}(1+q^{2n+2})q^{2n^{2}+3n+j^{2}},\nonumber \\
\sum_{n=0}^{\infty}\frac{q^{n^{2}+3n}}{(1-q^{2n+1})(q^{2};q^{2})_{n}} & =\frac{(-q^{2};q^{2})_{\infty}}{(q^{2};q^{2})_{\infty}}\sum_{n=0}^{\infty}\sum_{j=-n}^{n+1}(-1)^{n}(1-q^{2n+2})q^{2n^{2}+3n+j^{2}}.\label{eq:5-5}
\end{align}

\begin{thm}
For $\max\{|aq^{2}|,|bq^{2}|,|ab|\}<1,$ we have
\begin{equation}
\begin{aligned} & \frac{(q,abq;q)_{\infty}}{(aq^{2},bq^{2};q)_{\infty}}\text{}_{2}\phi_{1}\left(\begin{matrix}q/a,q/b\\
-q
\end{matrix};q,ab\right)\\
 & =\sum_{n=0}^{\infty}\sum_{j=-n}^{n+1}(-1)^{j}\frac{(1-q^{2n+2})(q/a,q/b;q)_{n}(ab)^{n}q^{j^{2}}}{(aq^{2},bq^{2};q)_{n}}.
\end{aligned}
\label{eq:7-9}
\end{equation}
\end{thm}
\begin{proof}
Taking $\alpha=q,c=-1$ in \eqref{eq:7-12} gives
\begin{equation}
(1-q^{n+1})q^{n(n+1)/2}{}_{2}\phi_{1}\left(\begin{matrix}q^{-n},q^{n+2}\\
-q
\end{matrix};q,1\right)=\sum_{j=-n}^{n+1}(-1)^{n+j}q^{j^{2}}.\label{eq:7-8}
\end{equation}
Setting $\alpha=q^{2},c=-q,z=q^{-1}$ in \eqref{eq:t2-9} yields
\begin{align*}
 & \frac{(q^{3},abq;q)_{\infty}}{(aq^{2},bq^{2};q)_{\infty}}\text{}_{2}\phi_{1}\left(\begin{matrix}q/a,q/b\\
-q
\end{matrix};q,ab\right)\\
 & =\sum_{n=0}^{\infty}\frac{(1-q^{2n+2})(1-q^{n+1})(q/a,q/b;q)_{n}(-ab)^{n}q^{n(n+1)/2}}{(1-q)(1-q^{2})(aq^{2},bq^{2};q)_{n}}\text{}_{2}\phi_{1}\left(\begin{matrix}q^{-n},q^{n+2}\\
-q
\end{matrix};q,1\right).
\end{align*}
Then the identity \eqref{eq:7-9} follows readily by substituting
\eqref{eq:7-8} into the right-hand side of the above formula.
\end{proof}
Certain Hecke-type identities associated with definite quadratic forms
can be obtained from the identity \eqref{eq:7-9}. Setting $(a,b)=(\pm1,0)$
and $(0,0)$ in \eqref{eq:7-9} we deduce
\begin{align*}
 & \sum_{n=0}^{\infty}\frac{(-1)^{n}q^{n(n+1)/2}}{(-q;q)_{n}}=\sum_{n=0}^{\infty}\sum_{j=-n}^{n+1}(-1)^{n+j}(1+q^{n+1})q^{n(n+1)/2+j^{2}},\\
 & \sum_{n=0}^{\infty}\frac{q^{n(n+1)/2}}{(q;q)_{n}}=\frac{(-q;q)_{\infty}}{(q;q)_{\infty}}\sum_{n=0}^{\infty}\sum_{j=-n}^{n+1}(-1)^{j}(1-q^{n+1})q^{n(n+1)/2+j^{2}},\\
 & \sum_{n=0}^{\infty}\frac{q^{n^{2}+n}}{(q^{2};q^{2})_{n}}=\frac{1}{(q;q)_{\infty}}\sum_{n=0}^{\infty}\sum_{j=-n}^{n+1}(-1)^{j}(1-q^{2n+2})q^{n^{2}+n+j^{2}}.
\end{align*}

\begin{thm}
For $\max\{|aq^{2}|,|bq^{2}|,|ab/q^{2}|\}<1,$ we have 
\begin{equation}
\begin{aligned} & \frac{(q^{2},ab;q^{2})_{\infty}}{(aq^{2},bq^{2};q^{2})_{\infty}}\text{}_{2}\phi_{1}\left(\begin{matrix}q^{2}/a,q^{2}/b\\
q
\end{matrix};q^{2},\frac{ab}{q^{2}}\right)\\
 & =\sum_{n=0}^{\infty}\sum_{j=-n}^{n}\frac{(1-q^{4n+2})(q^{2}/a,q^{2}/b;q^{2})_{n}(ab)^{n}}{(aq^{2},bq^{2};q^{2})_{n}}q^{2j^{2}+j-2n}.
\end{aligned}
\label{eq:7-11}
\end{equation}
\end{thm}
\begin{proof}
Replacing $q$ by $q^{2}$ and then setting $\alpha=1,c=q$ in \eqref{eq:7-12}
gives
\begin{equation}
\begin{aligned} & (-1)^{n}q^{n(n+1)}{}_{2}\phi_{1}\left(\begin{matrix}q^{-2n},q^{2n+2}\\
q
\end{matrix};q^{2},1\right)\\
 & =1+\sum_{j=1}^{n}(1+q^{2j})q^{2j^{2}-j}=\sum_{j=-n}^{n}q^{2j^{2}+j}.
\end{aligned}
\label{eq:7-10}
\end{equation}
Replacing $q$ by $q^{2}$ and putting $\alpha=q^{2},c=q,z=q^{-2}$
in \eqref{eq:t2-9} yields
\begin{align*}
 & \frac{(q^{4},ab;q^{2})_{\infty}}{(aq^{2},bq^{2};q^{2})_{\infty}}\text{}_{2}\phi_{1}\left(\begin{matrix}q^{2}/a,q^{2}/b\\
q
\end{matrix};q^{2},\frac{ab}{q^{2}}\right)\\
 & =\sum_{n=0}^{\infty}\frac{(1-q^{4n+2})(q^{2}/a,q^{2}/b;q^{2})_{n}(-ab)^{n}q^{n^{2}-n}}{(1-q^{2})(aq^{2},bq^{2};q^{2})_{n}}\text{}_{2}\phi_{1}\left(\begin{matrix}q^{-2n},q^{2n+2}\\
q
\end{matrix};q^{2},1\right).
\end{align*}
Then the formula \eqref{eq:7-11} follows easily by substituting \eqref{eq:7-10}
into the right-hand side of the above identity.
\end{proof}
Some Hecke-type identities associated with definite quadratic forms
can be established from the formula \eqref{eq:7-11}. Let\textbf{
$(a,b)=(\pm1,0),(\pm q,0)$ }and\textbf{ $(0,0)$ }in \eqref{eq:7-11}.
We deduce\textbf{ }
\begin{align*}
 & \sum_{n=0}^{\infty}\frac{(-1)^{n}q^{n^{2}-n}}{(q;q^{2})_{n}}=\sum_{n=0}^{\infty}\sum_{j=-n}^{n}(-1)^{n}(1-q^{4n+2})q^{n^{2}-n+2j^{2}+j},\\
 & \sum_{n=0}^{\infty}\frac{(-q^{2};q^{2})_{n}q^{n^{2}-n}}{(q;q)_{2n}}=\frac{(-q^{2};q^{2})_{\infty}}{(q^{2};q^{2})_{\infty}}\sum_{n=0}^{\infty}\sum_{j=-n}^{n}(1-q^{4n+2})q^{n^{2}-n+2j^{2}+j},\\
 & \sum_{n=0}^{\infty}\frac{(-1)^{n}q^{n^{2}}}{(q^{2};q^{2})_{n}}=\frac{(q;q^{2})_{\infty}}{(q^{2};q^{2})_{\infty}}\sum_{n=0}^{\infty}\sum_{j=-n}^{n}(-1)^{n}(1+q^{2n+1})q^{n^{2}+2j^{2}+j},\\
 & \sum_{n=0}^{\infty}\frac{(-q;q^{2})_{n}q^{n^{2}}}{(q;q)_{2n}}=\frac{(-q;q^{2})_{\infty}}{(q^{2};q^{2})_{\infty}}\sum_{n=0}^{\infty}\sum_{j=-n}^{n}(1-q^{2n+1})q^{n^{2}+2j^{2}+j},
\end{align*}
and
\[
\sum_{n=0}^{\infty}\frac{q^{2n^{2}}}{(q;q)_{2n}}=\frac{1}{(q^{2};q^{2})_{\infty}}\sum_{n=0}^{\infty}\sum_{j=-n}^{n}(1-q^{4n+2})q^{2n^{2}+2j^{2}+j}.
\]

The third identity above gives a second Hecke-type series representation,
which is associated with definite quadratic forms, for the series
$\sum_{n=0}^{\infty}\frac{q^{n^{2}}}{(q^{2};q^{2})_{n}}$ if $q$
is replaced by $-q$ while Liu \cite[(4.12)]{L6} presented a Hecke-type
series representation for that series, which is also stated in the
first section.

\section{Applications}

In this section we will use some Hecke-type identities associated
with definite quadratic forms in Section \ref{sec:3} to establish
inequalities for certain partition functions.

A partition of a positive integer $n$ is a non-increasing sequence
$(\lambda_{1},\lambda_{2},\cdots,\lambda_{r})$ of positive integers
such that $\lambda_{1}+\lambda_{2}+\cdot\cdot\cdot+\lambda_{r}=n.$
Let $p(n)$ and $\mathrm{pod}(n)$ denote the number of partitions
of $n$ and the number of partitions of $n$ without repeated odd
parts respectively and let $p(0)=\mathrm{pod(0)=1}$. Then the generating
functions for $p(n)$ and $\mathrm{pod}(n)$ are given by
\[
\sum_{n=0}^{\infty}p(n)q^{n}=\frac{1}{(q;q)_{\infty}}
\]
and 
\begin{equation}
\sum_{n=0}^{\infty}\mathrm{pod}(n)q^{n}=\frac{(-q;q^{2})_{\infty}}{(q^{2};q^{2})_{\infty}}.\label{eq:5-9}
\end{equation}

An overpartition of a positive integer $n$ is a partition of $n$
where the first occurrence of each distinct part may be overlined.
Let $\overline{p}(n)$ denote the number of overpartitions of $n$
and set $\overline{p}(0)=1$. Then the generating function for this
partition function is given by
\begin{equation}
\sum_{n=0}^{\infty}\overline{p}(n)q^{n}=\frac{(-q;q)_{\infty}}{(q;q)_{\infty}}.\label{eq:5-6}
\end{equation}

We now employ \eqref{eq:5-8}, \eqref{eq:5-1}, \eqref{eq:5-2}, \eqref{eq:5-3},
\eqref{eq:5-4} and \eqref{eq:5-5} to prove the following inequalities
for these partition functions.
\begin{thm}
\label{t4-1} For any nonnegative integer $N$ we have
\begin{align}
 & \sum_{n=0}^{\left\lfloor \sqrt{N}\right\rfloor }\sum_{j=-n}^{n}(-1)^{n}\overline{p}(N-n^{2}-n-j^{2})\geq0,\nonumber \\
 & \sum_{n=0}^{\left\lfloor \sqrt{N/2}\right\rfloor }\sum_{j=-n}^{n}(-1)^{n}\mathrm{pod}(N-2n^{2}-2n-j^{2})\geq0,\nonumber \\
 & \sum_{n=0}^{\left\lfloor \sqrt{N/2}\right\rfloor }\sum_{j=-n}^{n}(-1)^{n}\bigg(\overline{p}\bigg(\frac{N-n-j^{2}}{2}-n^{2}\bigg)+\overline{p}\bigg(\frac{N-3n-j^{2}-1}{2}-n^{2}\bigg)\bigg)\geq0,\label{eq:4-1}\\
 & \sum_{n=0}^{\left\lfloor \sqrt{N/3}\right\rfloor }\sum_{j=-n}^{n}(-1)^{n}\bigg(p\bigg(\frac{N-3n^{2}-2n-j^{2}}{2}\bigg)+p\bigg(\frac{N-3n^{2}-4n-j^{2}-1}{2}\bigg)\bigg)\geq0,\nonumber \\
 & \sum_{n=0}^{\left\lfloor \sqrt{N/2}\right\rfloor }\sum_{j=-n}^{n+1}(-1)^{n}\bigg(\overline{p}\bigg(\frac{N-3n-j^{2}}{2}-n^{2}\bigg)-\overline{p}\bigg(\frac{N-5n-j^{2}-2}{2}-n^{2}\bigg)\bigg)\geq0,\nonumber 
\end{align}
and
\begin{equation}
\sum_{n=0}^{\left\lfloor \sqrt{N/2}\right\rfloor }\sum_{j=-n}^{n}(-1)^{j}\mathrm{(pod}(N-2n^{2}-n-j^{2})-\mathrm{pod}(N-2n^{2}-3n-j^{2}-1))\geq0,\label{eq:5-10}
\end{equation}
where $p(x)=\mathrm{pod}(x)=\overline{p}(x)=0$ if $x$ is not a non-negative
integer and $\left\lfloor y\right\rfloor $ denotes the integral part
of a real number $y$.
\end{thm}
\begin{proof}
As the proofs of the first five inequalities are similar, we only
give the proofs of the inequalities \eqref{eq:4-1} and \eqref{eq:5-10}
here.

We first prove the inequality \eqref{eq:4-1}. It follows from \eqref{eq:5-6}
that 
\begin{equation}
\sum_{n=0}^{\infty}\overline{p}(n)q^{2n}=\frac{(-q^{2};q^{2})_{\infty}}{(q^{2};q^{2})_{\infty}}.\label{eq:5-7}
\end{equation}
The right-hand side of \eqref{eq:5-3} can be re-written as
\[
\frac{(-q^{2};q^{2})_{\infty}}{(q^{2};q^{2})_{\infty}}\sum_{n=0}^{\infty}\sum_{j=-n}^{n}(-1)^{n}(q^{2n^{2}+n+j^{2}}+q^{2n^{2}+3n+j^{2}+1}).
\]
Substituting \eqref{eq:5-7} into the above expression we conclude
that for each nonnegative integer $N$, the coefficient of $q^{N}$
on the right-hand side of \eqref{eq:5-3} is
\[
\sum_{n=0}^{\left\lfloor \sqrt{N/2}\right\rfloor }\sum_{j=-n}^{n}(-1)^{n}\bigg(\overline{p}\bigg(\frac{N-n-j^{2}}{2}-n^{2}\bigg)+\overline{p}\bigg(\frac{N-3n-j^{2}-1}{2}-n^{2}\bigg)\bigg).
\]
From the left-hand side of \eqref{eq:5-3} we find that the the coefficient
of $q^{N}$ for each nonnegative integer $N$ is nonnegative. Then
the result \eqref{eq:4-1} follows by comparing the coefficient of
$q^{N}$ on both sides of \eqref{eq:5-3}.

We now show the inequality \eqref{eq:5-10}. Replacing $q$ by $-q$
in \eqref{eq:5-8} yields

\begin{equation}
\sum_{n=0}^{\infty}\frac{q^{n^{2}+2n}}{(1+q^{2n})(q^{2};q^{2})_{n}}=\frac{(-q;q^{2})_{\infty}}{2(q^{2};q^{2})_{\infty}}\sum_{n=0}^{\infty}\sum_{j=-n}^{n}(-1)^{j}(1-q^{2n+1})q^{2n^{2}+n+j^{2}}.\label{eq:5-11}
\end{equation}
Since 
\[
\sum_{n=0}^{\infty}\frac{q^{n^{2}+2n}}{(1+q^{2n})(q^{2};q^{2})_{n}}=\frac{1}{2}+\sum_{n=1}^{\infty}\frac{q^{n^{2}+2n}}{(1-q^{4n})(q^{2};q^{2})_{n-1}},
\]
we see that for each nonnegative integer $N$, the coefficient of
$q^{N}$ on the left-hand side of \eqref{eq:5-11} is nonnegative.
Thus the result \eqref{eq:5-10} can be obtained by substituting \eqref{eq:5-9}
into the right-hand side of \eqref{eq:5-11} and then applying the
nonnegativity of the left-hand side of \eqref{eq:5-11}. This finishes
the proof of Theorem \ref{t4-1}.
\end{proof}
\begin{thm}
For any nonnegative integer $N$ we have
\begin{align}
 & \sum_{n=0}^{\left\lfloor N/2\right\rfloor }\sum_{j=-n}^{n}(-1)^{n}p(N-2n^{2}-2n-j^{2})\geq0,\label{eq:5-13}\\
 & \sum_{n=0}^{\left\lfloor N/2\right\rfloor }\sum_{j=-n}^{n}(-1)^{n}(p(N-2n^{2}-n-j^{2})+p(N-2n^{2}-3n-1-j^{2}))\geq0,\nonumber \\
 & \sum_{n=0}^{\left\lfloor N/3\right\rfloor }\sum_{j=-n}^{n}(-1)^{n}(p(N-3n^{2}-2n-j^{2})+p(N-3n^{2}-4n-1-j^{2}))\geq0,\nonumber \\
 & \sum_{n=0}^{\left\lfloor N/2\right\rfloor }\sum_{j=-n}^{n+1}(-1)^{n}(p(N-2n^{2}-3n-j^{2})-p(N-2n^{2}-5n-2-j^{2}))\geq0.\nonumber 
\end{align}
\end{thm}
\begin{proof}
As the proofs of these inequalities are similar, we only prove \eqref{eq:5-13}
here. Multiply both sides of \eqref{eq:5-2} by $(-q^{2};q^{2})_{\infty}$.
We have
\begin{equation}
(-q^{2};q^{2})_{\infty}\sum_{n=0}^{\infty}\frac{(-q;q^{2})_{n}q^{n^{2}+2n}}{(1-q^{2n+1})(q^{4};q^{4})_{n}}=\frac{1}{(q;q)_{\infty}}\sum_{n=0}^{\infty}\sum_{j=-n}^{n}(-1)^{n}q^{2n^{2}+2n+j^{2}}.\label{eq:5-12}
\end{equation}
It is easy to see that the left-hand side of \eqref{eq:5-12} has
nonnegative coefficients. Then the inequality \eqref{eq:5-13} follows
easily by equating the coefficients of $q^{N}$ on both sides of \eqref{eq:5-12}.
\end{proof}

\section*{Acknowledgement}

 This work was partially supported by the National Natural Science
Foundation of China (Grant No. 11801451).

\end{document}